%% file: model.tex
\documentclass[reqno,a4paper,oneside]{amsart}

\input{uniform-kan-prelude}

\input{formatting-christian}

\title[The Equivalence Extension Property and Model Structures]{The Equivalence Extension Property \\ and Model Structures}

\author{Christian Sattler}
\email{sattler@chalmers.se}

\newcommand{\C}{\mathbf{C}}
\newcommand{\TC}{\mathbf{TC}}
\newcommand{\F}{\mathbf{F}}
\newcommand{\TF}{\mathbf{TF}}
\newcommand{\W}{\mathbf{W}}

\begin{document}

\begin{abstract}
We give an elementary construction of a certain class of model structures.
In particular, we rederive the Kan model structure on simplicial sets without the use of topological spaces, minimal complexes, or any concrete model of fibrant replacement such as Kan's $\Ex^\infty$ functor.
Our argument makes crucial use of the glueing construction developed by Cohen \etal~\cite{cohen-et-al:cubicaltt} in the specific setting of certain cubical sets.
\end{abstract}

\maketitle

\tableofcontents

\section{Introduction}

The goals of this paper are twofold.

First, we continue the programme of~\cite{gambino-sattler:frobenius}, which gives a categorical analysis of the Frobenius condition, by giving a categorical analysis of the glueing construction of~\cite{cohen-et-al:cubicaltt}.
This construction was originally developed to facilitate a proof of fibrancy and univalence of universes in a cubical set model of homotopy type theory.
It allows one to extend an equivalence between fibrations along a cofibration, given an extension of one of the fibrations (extension meaning forming a cartesian square).
In order to avoid the overloaded term ``glueing'', we call this the equivalence extension property (the term was suggested by Steve Awodey).
The correspondence to univalence in the setting of certain model structures is detailed in~\cite{voevodsky-simplicial-model,cisinski-univalence}.

Second, culminating in \cref{main-theorem}, we show how the construction can also be used for the construction of certain Quillen model structures based on functorial cylinders.
As our main application in \cref{kan-model-structure}, we obtain an elementary proof of the Kan model structure on simplicial sets~\cite{quillen-homotopical} that does not make use of topological spaces or minimal complexes, in contrast to the proofs in \cite{methods-of-homological-algebra,joyal-tierney-notes,joyal-tierney:simplicial-homotopy-theory}, and additionally does not depend on combinatorics of a specific model of fibrant replacement such as Kan's $\Ex^\infty$ functor, in contrast to the proof in~\cite{cisinski-asterisque}.

For the reader interested in constructive aspects, we note that our construction of the Kan model structure avoids the axiom of choice.
The non-constructivity is now neatly encapsulated in the property of the simplex category $\Delta$ as an elegant Reedy category that any monomorphism in simplicial sets can be written as an $\omega$-composition of cobase changes of coproducts of boundary inclusions of simplices.
This part critically requires the axiom of excluded middle to decide whether an element of a simplicial set is degenerate.

A sequel to this paper, continuing the programme of~\cite{gambino-sattler:frobenius} also for uniform notions of fibrations part of algebraic weak factorization systems, is currently in preparation.
It will provide an actual generalization of the results of~\cite{cohen-et-al:cubicaltt} to an abstract setting.
This will be used to construct certain algebraic model structures~\cite{riehl-algebraic-model} in the stronger sense of~\cite{swan:ams} using constructive methods, and as a corollary yield algebraic model structures on certain categories of cubical sets that are complete in the sense of Cisinski~\cite{cisinski-asterisque}.

\subsection*{Organization of the paper}

This paper is split into two parts.
In the first, consisting of only \cref{section:criterion}, we will develop the sufficient criteria \cref{model-structure,right-proper-model-structure} for when two weak factorization systems (wfs's) give rise to a (right proper) model structure, introducing what call the extension property.

The second part, and the paper proper, starts afterwards.
\cref{section:setting} introduces the setting we will be working in, consisting of a category $\cal{E}$, a functorial cyclinder, and a wfs $(\C, \TF)$ satisfying conditions we deem suitable.
\cref{section:preliminary} constructs the wfs $(\TC, \F)$ from that data, listing two more assumptions~\ref{boundary-fibration} and~\ref{local} it needs to satisfy, and develops several preliminary notions, all of them well-known from abstract homotopy theory.
In \cref{section:glueing}, we present the proof of the equivalence extension property, the central technical aspect of our development.
In \cref{section:composition}, we discuss an alternative way of characterizing fibrations in terms of lifts against squares instead of arrows.
This is used in \cref{section:extension} when applying the equivalence extension property to derive the extension property of fibrations along trivial cofibrations.

With this, in \cref{section:result} we can finally apply the criterion developed in the first part of the paper to construct the model structure in the form of \cref{main-theorem}; we finish by discussing important examples such as simplicial sets and certain categories of cubical sets.

\subsection*{Acknowledgements}

We thank Thierry Coquand for support of a visit of the author to Gothenburg in November~2015, which led to many interesting discussions.
Several key ideas underlying this paper were developed in a subsequent email exchange in December~2015 with Thierry Coquand and Andrew Swan.
We thank Nicola Gambino for discussion of ideas, proofreading of various draft documents, and comments on organization.
We thank Simon Huber and Andrew Swan for discussions on combinatorial aspects of different notions of fibrations in various variations of cubical sets.
We thank Jonas Frey for spotting a careless error in an earlier write-up of the proof of \cref{glueing-core}.

This material is based on research sponsored by the Air Force Research Laboratory, under agreement number FA8655-13-1-3038.

\section{A criterion for model structures}
\label{section:criterion}

A \emph{model structure} on a category $\cal{E}$ consists of three classes of maps $\C, \W, \F$ such that $(\C, \F \cap \W)$ and $(\C \cap \W, \F)$ form weak factorization systems (wfs's) and $\W$ satisfies 2-out-of-3.
Compared to \cite[Definition~1.1.3]{hovey:model-categories}, we do not ask the wfs's to come with functorial factorizations; note that the classes $\C$ and $\F$ are vacuously closed under retracts since they are respectively part of a wfs, and the same holds for $\W$ as proved in \cite[Proposition~E.1.3]{joyal-quaderns}.

Let $\cal{E}$ be a finitely complete and cocomplete category with two wfs's $(\C, \TF)$ and $(\TC, \F)$ such that $\TC \subseteq \C$ (equivalently, $\TF \subseteq \F$).
In this section, we will develop a simple sufficient criterion for this data to form a model structure (note that the class $\W$ is determined by $\C$ and $\F$).
This criterion is far from necessary, but it will be satisfied for the model structures modelling $\omega$-groupoids on simplicial sets and cubical sets over certain cube categories.

We call the maps in $\C$ ($\TC$) (trivial) cofibrations and draw them $A \cof B$ ($A \trivcof B$).
We call the maps in $\F$ ($\TF$) (trivial) fibrations and draw them $Y \fib X$ ($Y \trivfib X$).
We define a map to be a \emph{weak equivalence}, drawn $A \we B$, if it factors as a trivial cofibration followed by a trivial fibration.
The class of weak equivalences is denoted $\W$.
We have the following standard result.

\begin{lemma} \label{classes-match}
We have $\TF = \F \cap \W$ and $\TC = \C \cap \W$.
\end{lemma}

\begin{proof}
We have $\TC \subseteq \C$ and $\TF \subseteq \F$ by assumption.
Note that $\TC, \TF \subseteq \W$ since identities belong to $\TF$ and $\TC$, respectively.

The other directions follow from a standard retract argument.
For example, given a fibration $Y \trivfib X$ that factors as a trivial cofibration $Y \trivcof M$ followed by a trivial fibration $M \trivfib X$, we have a lifting diagram as follows:
\[
\xymatrix{
  Y
  \ar@{=}[r]
  \ar@{>->}[d]_{\sim}
&
  Y
  \ar@{->>}[d]
\\
  M
  \ar@{->>}[r]^{\sim}
  \ar@{.>}[ur]
&
  X
\rlap{.}}
\]
The lift exhibits $Y \to X$ as a domain retract of $M \to X$.
By closure of trivial fibrations under domain retracts, this makes $Y \to X$ into a trivial cofibration.
\end{proof}

We now list some conditions we are going to consider for our criterion.

\begin{definition}[Span property] \label{span-property}
We have the \emph{span property} if in any commuting triangle
\[
\xymatrix{
&
  A
  \ar@{>->}[dl]_{\triv}
  \ar@{>->}[dr]^{\triv}
\\
  Y
  \ar@{->>}[rr]
&&
  X
\rlap{,}}
\]
the map $Y \to X$ is a trivial fibration.
\end{definition}

\begin{definition}[Exchange] \label{exchange}
A class of maps $\mathbf{B}$ has \emph{exchange} along a class of maps $\mathbf{A}$ if for maps $X \to A$ in $\mathbf{B}$ and $A \to B$ in $\mathbf{A}$, there are maps $Y \to B$ in $\mathbf{B}$ and $X \to Y$ in $\mathbf{A}$ forming a commuting square as follows:
\[
\xymatrix{
  X
  \ar@{.>}[r]^{\in \mathbf{A}}
  \ar[d]_{\in \mathbf{B}}
&
  Y
  \ar@{.>}[d]^{\in \mathbf{B}}
\\
  A
  \ar[r]_{\in \mathbf{A}}
&
  B
\rlap{.}}
\]
It has \emph{cartesian exchange} if the above square is in addition a pullback.
\end{definition}

We will develop our criterion through as series of lemmata.
As a start, standard reasoning shows the following.

\begin{lemma} \label{factorization-invariance}
Assume that $\cal{E}$ has the span property and that trivial fibrations satisfy 2-out-of-3 relative to (\ie, in the subcategory of) fibrations.
Let
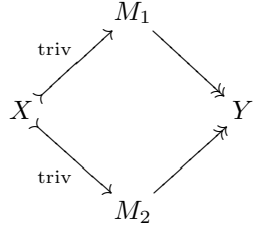
\[
\xymatrix{
&
  M_1
  \ar@{->>}[dr]
\\
  X
  \ar@{>->}[ur]^{\triv}
  \ar@{>->}[dr]_{\triv}
&&
  Y
\\&
  M_2
  \ar@{->>}[ur]
}
\]
be two factorizations of a map $X \to Y$ into a trivial cofibration followed by a fibration.
If $M_1 \to Y$ is a trivial fibration, then so is $M_2 \to Y$.
\end{lemma}

\begin{proof}
We introduce the pullback $P$ of $M_1 \to Y$ and $M_2 \to Y$.
Since fibrations are closed under pullback, we have that $P \to M_1$ and $P \to M_2$ are fibrations.
We then factor the induced map $X \to P$ into a trivial cofibration followed by a fibration:
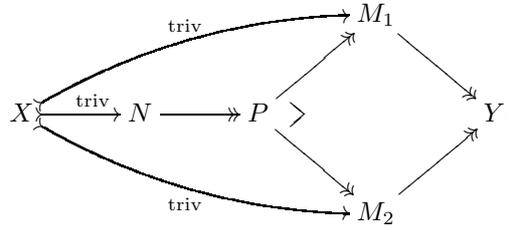
\[
\xymatrix@!C{
&&&
  M_1
  \ar@{->>}[dr]
\\
  X
  \ar@{>->}@/^1em/[urrr]^{\triv}
  \ar@{>->}@/_1em/[drrr]_{\triv}
  \ar@{>->}[r]^(0.6){\triv}
&
  N
  \ar@{->>}[r]
&
  P
  \ar@{->>}[ur]
  \ar@{->>}[dr]
  \fancypullback{[ur]}{[dr]}{[rr]}
&&
  Y
  \rlap{.}
\\&&&
  M_2
  \ar@{->>}[ur]
}
\]
Since fibrations are stable under compositions, we have that $N \to M_1$ and $N \to M_2$ are fibrations.
By the span property, they are trivial fibrations.

Now let $M_1 \to Y$ be a trivial fibration.
By closure under base change, then so is $P \to M_2$.
By various instances of 2-out-of-3 for trivial fibrations relative to fibrations, we first have $N \to P$, then $P \to M_1$, and finally $M_2 \to Y$ a trivial fibration.
\end{proof}

\begin{corollary} \label{weak-equiv-vs-trivial-fib}
Under the assumptions of \cref{factorization-invariance}, in any commuting diagram
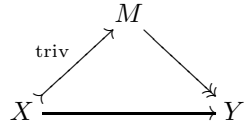
\[
\xymatrix{
&
  M
  \ar@{->>}[dr]
\\
  X
  \ar@{>->}[ur]^{\triv}
  \ar[rr]
&&
  Y
}
\]
the map $X \to Y$ is a weak equivalence if and only if $M \to Y$ is a trivial fibration.
\qed
\end{corollary}

Note that this gives a definition of weak equivalence that does not require a quantification over all possible $(\TC, \TF)$-factorizations.
Rather, it is enough to consider a single $(\TC, \F)$-factorization.

We now have the following sufficient criterion for the given weak factorization systems to form a model structure.

\begin{theorem} \label{model-structure}
Assume that the following assumptions are satisfied:
\begin{enumerate}[label=(\roman*)]
\item the span property holds,
\item trivial fibrations satisfy 2-out-of-3 relative to (\ie in the subcategory of) fibration,
\item trivial fibrations have exchange with trivial cofibrations,
\item fibrations have cartesian exchange with trivial cofibrations.
\end{enumerate}
Then $(\C, \W, \F)$ forms a model structure.
\end{theorem}

\begin{proof}
In view of \cref{classes-match}, it only remains to verify that the class $\W$ satisfies 2-out-of-3.
Consider a commuting triangle as follows:
\[
\xymatrix{
  X
  \ar[rr]
  \ar[dr]
&&
  Z
  \rlap{.}
\\&
  Y
  \ar[ur]
}
\]
We need to show: if two of these maps are weak equivalences, then so is the third.
For this, we factor each of $X \to Y$ and $Y \to Z$ into a trivial cofibration followed by a fibration:
\begin{equation} \label{model-structure:1}
\begin{gathered}
\xymatrix{
  X
  \ar@{>->}[r]^{\triv}
&
  U
  \ar@{->>}[d]
\\&
  Y
  \ar@{>->}[r]^{\triv}
&
  V
  \ar@{->>}[d]
\\&&
  Z
\rlap{.}}
\end{gathered}
\end{equation}

Let us first deal with the cases where $X \to Y$ is a weak equivalence.
We can then have $U \to Y$ in~\eqref{model-structure:1} a trivial fibration.
We use the exchange property for $\TC$ and $\TF$ to extend $U \trivfib Y$ along $Y \trivcof V$ as below:
\[
\xymatrix{
  X
  \ar@{>->}[r]^{\triv}
&
  U
  \ar@{->>}[d]_{\triv}
  \ar@{>.>}[r]^{\triv}
&
  M
  \ar@{.>>}[d]_{\triv}
\\&
  Y
  \ar@{>->}[r]_{\triv}
&
  V
  \ar@{->>}[d]
\\
&&
  Z
\rlap{.}}
\]
By \cref{weak-equiv-vs-trivial-fib}, $X \to Z$ is a weak equivalence if and only if $M \to Z$ is a trivial fibration.
Similarly, $Y \to Z$ is a weak equivalence if and only if $V \to Z$ is a trivial fibration.
By 2-out-of-3 for trivial fibrations among fibrations, these two assertions are equivalent.

Let us now deal with the case where $X \to Z$ and $Y \to Z$ are weak equivalences.
We can then have $V \to Z$ in~\eqref{model-structure:1} a trivial fibration.
We use the cartesian extension property for $\TC$ and $\F$ to extend $U \fib Y$ along $Y \trivcof V$ as below:
\[
\xymatrix{
  X
  \ar@{>->}[r]^{\triv}
&
  U
  \ar@{->>}[d]
  \ar@{>.>}[r]^{\triv}
  \fancypullback{[d]}{[r]}{[dr]}
&
  M
  \ar@{.>>}[d]
\\&
  Y
  \ar@{>->}[r]_{\triv}
&
  V
  \ar@{->>}[d]_{\triv}
\\
&&
  Z
\rlap{.}}
\]
By \cref{weak-equiv-vs-trivial-fib}, $M \to Z$ is a trivial fibration.
By 2-out-of-3 for trivial fibrations among fibrations, $M \to V$ is a trivial fibration.
By closure under pullback, $U \to Y$ is a trivial fibration, making $X \to Y$ into a weak equivalence.
\end{proof}

We next state a simplified version of this criterion in case $(\TC, \F)$ satisfies the Frobenius property~\cite{garner:topological-simplicial,gambino-garner:idtypewfs}, \ie that trivial cofibrations are preserved under pullback along fibrations.
For this, we need the following notion.

\begin{definition}[Extension] \label{extension}
A class of maps $\mathbf{B}$ has \emph{extension} along a class of maps $\mathbf{A}$ if for maps $X \to A$ in $\mathbf{B}$ and $A \to B$ in $\mathbf{A}$, there are maps $Y \to B$ in $\mathbf{B}$ and $X \to Y$ forming a pullback square as follows:
\[
\xymatrix{
  X
  \ar@{.>}[r]
  \ar[d]_{\in \mathbf{B}}
  \fancypullback{[d]}{[r]}{[dr]}
&
  Y
  \ar@{.>}[d]^{\in \mathbf{B}}
\\
  A
  \ar[r]_{\in \mathbf{A}}
&
  B
\rlap{.}}
\]
\end{definition}

\begin{theorem} \label{right-proper-model-structure}
Assume that the following assumptions are satisfied:
\begin{enumerate}[label=(\roman*)]
\item the span property holds,
\item trivial fibrations satisfy 2-out-of-3 relative to (\ie in the subcategory of) fibration,
\item fibrations and trivial fibrations extend along trivial cofibrations,
\item the wfs $(\TC, \F$) satisfies the Frobenius property.
\end{enumerate}
Then $(\C, \W, \F)$ forms a right proper model structure.
\end{theorem}

\begin{proof}
With the Frobenius property, condition~(iii) implies conditions~(iii) and~(iv) of \cref{model-structure}.
Right properness follows from the Frobenius property.
\end{proof}

Note that extension of (trivial) fibrations along trivial cofibrations is also considered in~\cite[Lemma~2.17 and Proposition~2.21]{cisinski-univalence}, but there the right proper model structure is the starting point.
A related extension problem of what they call bundles along trivial cofibrations is considered in~\cite[Lemma~1.7.1]{joyal-tierney-notes}.

\section{Suitable setting}
\label{section:setting}

\subsection{Category}
\label{subsection:category}

We call a category $\cal{E}$ \emph{suitable} if it is locally presentable, locally cartesian closed, and (infinitary) extensive, \ie whose coproducts are van Kampen in the sense of~\cite{heindel:van-kampen-universal}.
Note that any Grothendieck topos is suitable.

Any morphism $f \co X \to Y$ gives rise to a pullback functor $f^* \co \cal{E}_{/Y} \to \cal{E}_{/X}$ with a left adjoint $f_! \co \cal{E}_{/X} \to \cal{E}_{/Y}$ and a right adjoint $f_* \co \cal{E}_{/X} \to \cal{E}_{/Y}$, called pushforward.
We write the exponential of an object $A$ with an object $B$ as $\hom(A, B)$.

Let $\Adj(\cal{E}, \cal{E})$ denote the category of adjunctions between endofunctors on $\cal{E}$.
Recall that the canonical functors $\Adj(\cal{E}, \cal{E}) \to [\cal{E}, \cal{E}]$ and $\Adj(\cal{E}, \cal{E}) \to [\cal{E}, \cal{E}]^\op$ are fully faithful.
We write $(-) \otimes (-)$ and $(-) \obackslash (-)$ for their uncurried versions, application of the left and right adjoint, respectively.
These infix operators are to be read right associative by default.
Often, we will denote an adjunction just by its left adjoint.
The reason for this choice of notation will become evident when discussing functorial cylinders.

We recall the Leibniz construction~\cite{riehl-verity:reedy}, which lifts any bifunctor $F \co \cal{A} \times \cal{B} \to \cal{C}$ to a bifunctor $\widehat{F} \co \cal{A}^\to \times \cal{B}^\to \to \cal{C}^\to$ between categories of arrows assuming that $\cal{C}$ has pushouts, and its properties.
We will apply to a variety of bifunctors.
Note that for the purpose of the Leibniz construction, we will consider exponential and right adjoint application to have signatures $\hom \co \cal{E} \times \cal{E}^\op \to \cal{E}^\op$ and $(-) \obackslash (-) \co [\cal{E}, \cal{E}] \times \cal{E}^\op \to \cal{E}^\op$, resulting in the use of pullbacks instead of pushouts in $\cal{E}$.

\subsection{Functorial cylinder}
\label{subsection:cylinder}

A \emph{suitable functorial cylinder} on a suitable category $\cal{E}$ is an endofunctor $I$ with endpoints inclusions $\delta_0, \delta_1 \co \Id \to I$, contractions $\epsilon \co I \to \Id$, and connections $c^0, c^1 \co I \circ I \to I$.
In addition to the laws imposed on these natural transformations in~\cite{gambino-sattler:frobenius}, we also require that the diagram
\begin{equation} \label{functorial-cylinder:contraction-connection}
\begin{gathered}
\xymatrix@!C@C-2em{
  I \circ I
  \ar[dr]_{\epsilon \cc \epsilon}
  \ar[rr]^-{c^k}
&&
  I
  \ar[dl]^{\epsilon}
\\&
  \Id
}
\end{gathered}
\end{equation}
commutes and that the endpoint inclusions are disjoint, \ie
\begin{equation} \label{functorial-cylinder:disjoint-endpoints}
\begin{gathered}
\xymatrix{
  0
  \ar[r]^{\bot_{\Id}}
  \ar[d]_{\bot_{\Id}}
  \fancypullback{[d]}{[r]}{[dr]}
&
  \Id
  \ar[d]^{\delta_0}
\\
  \Id
  \ar[r]^{\delta_1}
&
  I
}
\end{gathered}
\end{equation}
forms a pullback.
We assume that $I$ has a right adjoint, inducing a functorial cocylinder.
In line with the infix notation introduced before, we write $I \otimes (-)$ for the application of the functorial cylinder $I$ and $I \obackslash (-)$ for its right adjoint functorial cocylinder application.

A suitable functorial cylinder is for example induced by a (left or right) closed monoidal structure with an interval object that carries structure analogous to the one outlined above.

The added law~\eqref{functorial-cylinder:contraction-connection} is required to make the following result hold.

\begin{lemma} \label{cylinder-slicing}
The structure of a suitable functorial cylinder is stable under slicing.
In detail, given a suitable functorial cylinder $I$ on $\cal{E}$, then for any $X \in \cal{E}$ there is a suitable functorial cylinder $I_{/X}$ on $\cal{E}_{/X}$ defined as the composition
\[
\xymatrix@C+1em{
  \cal{E}_{/X}
  \ar[r]^-{I}
&
  \cal{E}_{/(I \otimes X)}
  \ar[r]^-{(\epsilon \otimes X)_!}
&
  \cal{E}_{/X}
\rlap{.}}
\]
The forgetful functor $\cal{E}_{/X} \to \cal{E}$ preserves all the structure of the functorial cylinder.
\end{lemma}

\begin{proof}
Standard diagram chasing.
The right adjoint to $I_{/X}$ is given by the right adjoint of $I$ followed by pullback along $\epsilon \otimes X$.
For~\eqref{functorial-cylinder:disjoint-endpoints}, recall that $\cal{E}_{/X} \to \cal{E}$ creates pullbacks.
\end{proof}

Note that monoidal structures are not stable under arbitrary slicing, giving one justification for our chosen level of abstraction.

We denote $i^1 \co \partial I \to I$ where $\partial I \defeq \Id + \Id$ and $i^1 \defeq [\delta_0, \delta_1]$ the \emph{boundary inclusion} of the functorial cylinder.
For convenience, we write $i^n \co \partial I^n \to I^n$ for the iterated Leibniz composition $i^n \defeq i^1 \hatcirc \cdots \hatcirc i^1$ with $n$ components.

We recall from~\cite{gambino-sattler:frobenius} the notions of homotopy, (strong) homotopy equivalence, and (strong) deformation retract induced by a functorial cylinder.

\subsection{Weak factorization system}
\label{subsection:wfs}

Let $\cal{E}$ be a suitable category with a suitable functorial cylinder.
Let $(\C, \TF)$ be a wfs in $\cal{E}$.
We call \emph{cofibrations} the elements of $\C$ and \emph{trivial fibrations} the elements of $\TF$.

\begin{definition} \label{suitable-wfs}
The wfs $(\C, \TF)$ is called \emph{suitable} if:
\begin{enumerate}[label=(\roman*)]
\item
it is cofibrantly generated,
\item
cofibrations are adhesive and exhaustive,
\item
cofibrations are closed under pullback,
\item
cofibrations are closed under finitary union,
\item
the functorial cylinder $I$ preserves cofibrations,
\item
the endpoint inclusions $\delta_0, \delta_1$ are valued in cofibrations.
\end{enumerate}
\end{definition}

For the notion of adhesiveness, we refer to~\cite{garner-lack:adhesive}.
Essentially, a map is adhesive if pushouts along it are van Kampen.
Exhaustiveness refers to the analogous notion for transfinite compositions introduced in~\cite{shulman:univalence-simplicial}, requiring these colimits to be van Kampen.

Note that the adhesiveness of conditions~(ii), the unions in condition~(iv) are hence automatically effective.

If $\cal{E}$ is a presheaf category and $I$ preserves monomorphisms, an example of a suitable wfs $(\C, \TF)$ is given by taking $\C$ to consist of all monomorphisms.
Cofibrant generation is proven in~\cite[Proposition~1.2.27]{cisinski-asterisque}.

\begin{lemma} \label{suitable-wfs-slicing}
Suitable wfs's are stable under slicing.
In detail, given a suitable wfs $(\C, \TF)$ on $\cal{E}$, then for any $X \in \cal{E}$ the induced wfs $(\C_{/X}, \TF_{/X})$ on $\cal{E}_{/X}$ is suitable as well.
\end{lemma}

\begin{proof}
The existence and cofibrant generation of $(\C_{/X}, \TF_{/X})$ on $\cal{E}_{/X}$ is classical, with the classes created by the forgetful functor $\cal{E}_{/X} \to \cal{E}$.
Note that $\cal{E}_{/X} \to \cal{E}$ creates colimits and pullbacks, hence reflects pre-adhesive morphisms.
It follows that conditions~(ii) to~(iv) hold for the wfs on $\cal{E}_{/X}$.
Conditions~(v) and~(vi) hold since $\cal{E}_{/X} \to \cal{E}$ preserves the structure of the functorial cylinder given by \cref{cylinder-slicing}.
\end{proof}

Note that, by the nullary case of condition~(iv) of \cref{suitable-wfs}, every object is cofibrant, \ie the map $\bot_X \co 0 \to X$ is a cofibration for $x \in \cal{E}$.
It follows that every trivial fibration has a section.

\begin{lemma} \label{pushout-product}
Under conditions~(ii) and~(iii) of \cref{suitable-wfs}, condition~(iv) is equivalent to stability of cofibrations under finitary pushout product.
\end{lemma}

\begin{proof}
For the forward direction, note that pushout product can be decomposed as base change of inputs followed by effective union.
For the reverse direction, note that effective union writes as pushout product followed by pullback along the diagonal.
\end{proof}

\begin{corollary} \label{cofibration-exp-trivial-fibrations}
Trivial fibrations are stable under Leibniz exponential with cofibrations.
\qed
\end{corollary}

\begin{lemma} \label{cofib-leibniz}
Let $u \co F \to G$ be a natural transformation between endofunctors on $\cal{E}$ such that $G$ preserves cofibrations.
Under conditions~(ii) and~(iv) of \cref{suitable-wfs}, the following are equivalent:
\begin{enumerate}[label=(\roman*)]
\item
$u$ is valued in cofibrations,
\item
Leibniz application of $u$ preserves cofibrations,
\end{enumerate}
and if $u$ has a right adjoint:
\begin{enumerate}[label=(\roman*),start=3]
\item
right adjoint Leibniz application of $u$ preserves trivial fibrations.
\end{enumerate}
\end{lemma}

\begin{proof}
From~(i) to~(ii), note that $u \hatotimes m$ is the effective union of $G \otimes m$ and $u \otimes B$ for any cofibration $m \co A \to B$.
From~(ii) to~(i), note that $u \otimes X = u \hatotimes \bot_X$ for $X \in \cal{E}$.
The equivalence of~(ii) and~(iii) follows from adjointness.
\end{proof}

\begin{corollary} \label{pushout-application}
Under conditions~(iv) and~(v) of \cref{suitable-wfs}, condition~(vi) is equivalent to stability of cofibrations under Leibniz application of endpoint inclusions.
Trivial fibrations are stable under right adjoint Leibniz application of endpoint inclusions.
\qed
\end{corollary}

\begin{corollary} \label{boundary-inclusion}
The boundary inclusion $i^1 \co \partial I \to I$ is valued in cofibrations.
Leibniz application of $i^1$ preserves cofibrations.
Right adjoint Leibniz application of $i^1$ preserves trivial fibrations.
\qed
\end{corollary}

Note that trivial fibrations are closed under pushforward along arbitrary maps, the adjoint formulation of condition~(iii) of \cref{suitable-wfs}.
Note also that cofibrations are monomorphisms by condition~(ii) of \cref{suitable-wfs}.
For a cofibration $m$, it follows that the adjunction $m^* \dashv m_*$ of pullback and pushforward along $m$ is a reflection.
This implies the following statement, referred to as Joyal's trick in~\cite{cisinski-univalence}.

\begin{lemma} \label{triv-fib-extend-along-cof}
Trivial fibrations extend along cofibrations in the sense of \cref{extension}.
\qed
\end{lemma}

\section{Preliminary notions}
\label{section:preliminary}

\subsection{Fibrations}
\label{subsection:fibrations}

We recall from~\cite{gambino-sattler:frobenius} how the wfs $(\C, \TF)$ gives rise to a second wfs $(\TC, \F)$ of \emph{trivial cofibrations} and \emph{fibrations}.
Let $\cal{I}$ be a set of generators for the suitable wfs $(\C, \TF)$, guaranteed to exist by condition~(i) of \cref{suitable-wfs}.
We let $\cal{J}$ be the set of Leibniz applications of endpoint inclusions to generating cofibrations, \ie
\begin{equation} \label{gen-triv-cof}
\cal{J} \defeq \braces{\delta_0, \delta_1} \hatotimes \cal{I} = \set{\delta_k \hatotimes m}{k \in \braces{0, 1}, m \in \cal{I}}
.\end{equation}
Since $\cal{E}$ is locally presentable, $\cal{J}$ generates a wfs $(\TC, \F)$.
Note that, by adjointness, a map $p$ is a fibration if and only if $\delta_k \hatobackslash p$ is a trivial fibration, \ie $\delta_k \hatotimes m \pitchfork p$ for all $m \in \C$, for $k \in \braces{0, 1}$.

\begin{remark} \label{fib-not-stable-under-slicing}
The construction of the wfs $(\TC, \F)$ from the wfs $(\C, \TF)$ does not commute with slicing, \ie \cref{cylinder-slicing,suitable-wfs-slicing} do not extend to the wfs $(\TC, \F)$.
Let us illuminate this subtlety.

Given a suitable wfs $(\C, \TF)$ and $X \in \cal{E}$, the induced wfs $(\C_{/X}, \TF_{/X})$ on $\cal{E}_{/X}$ is suitable as well by \cref{suitable-wfs-slicing}.
Let $(\TC, \F)$ and $(\TC', \F')$ denote the wfs's of trivial cofibrations and fibrations generated by $\cal{J}$ and $\cal{J}'$ as defined in~\eqref{gen-triv-cof} from $(\C, \TF)$ and $(\C_{/X}, \TF_{/X})$ in $\cal{E}$ and $\cal{E}_{/X}$, respectively.
The wfs $(\TC, \F)$ generated by $\cal{J}$ in $\cal{E}$ induces a wfs $(\TC_{/X}, \F_{/X})$ generated by $\cal{J}_{/X}$ in $\cal{E}_{/X}$.

We have $\cal{J}' \subseteq \cal{J}_{/X}$: Leibniz application in the slice forces the codomains of maps in $\cal{J}'$, of the form $I \otimes B \to X$ with $B \in \cal{E}$, to lift through $\epsilon \otimes X$.
It follows that $\TC' \subseteq \TC_{/X}$ and $\F_{/X} \subseteq \F'$, but the reverse inclusions do not hold in general.
In particular, a map in $\cal{E}_{/X}$ is in $\F'$ if its underlying map is a fibration, but the converse does not hold in general.

To resolve the double meaning of the notion of a fibration in the slice $\cal{E}_{/X}$, we will always mean an element of $\F_{/X}$ rather than $\F'$ in the rest of this document.
\end{remark}

Our current setting of a suitable category with a suitable functorial cylinder and a suitable wfs are stronger than the setting and notion of a suitable wfs in \cite{gambino-sattler:frobenius}.
We record the main result of its first part so that we may use it.

\begin{theorem}[Theorem~3.8 of~\cite{gambino-sattler:frobenius}] \label{frobenius}
The wfs $(\TC, \F)$ satisfies the Frobenius property.
\end{theorem}

We now make the following additional assumptions:
\begin{enumerate}[label=(A.\arabic*)]
\item \label{boundary-fibration}
trivial cofibrations are closed under Leibniz application of $i^1 \co \partial I \to I$.
\item \label{local}
fibrations are \emph{local}, in the sense that given a cartesian diagram $F$ in $\cal{E}^\to$ with a cartesian colimiting cocone with colimit $p$, if all objects in $F$ are fibrations, then so is $p$.
\end{enumerate}

We will see in \cref{path-objects} that assumption~\ref{boundary-fibration} is necessary to get well-behaved path objects.
Using adjointness, it is equivalent to any of the following conditions:
\begin{itemize}
\item
fibrations are closed under right adjoint Leibniz application of $i^1$,
\item
Leibniz application of $i^1$ maps generating trivial cofibrations (\ie elements of $\cal{J}$ as defined in~\eqref{gen-triv-cof}) to trivial cofibrations.
\end{itemize}

\begin{remark} \label{swap-arguments}
A sufficient condition for satisfying~\ref{boundary-fibration} is given by a symmetry of the functorial cylinder in the form of a natural isomorphism $I \circ I \cong I \circ I$ that coheres with its other structure in the evident way.
We can then show $i^1 \hatotimes \delta_k \cong \delta_k \hatotimes i^1$.
It follows that $\cal{J}$ as defined in~\eqref{gen-triv-cof} is already closed (up to isomorphism) under Leibniz application of $i^1$: for $m \in \C$, we have $i^1 \hatotimes \delta_k \hatotimes m \cong \delta_k \hatotimes i^1 \hatotimes m \in \cal{J}'$ as $i^1 \hatotimes m \in \C$ by \cref{boundary-inclusion}.
In that case, we have that $\cal{J}$ coincides with Cisinski's generators for \emph{naive fibrations}~\cite{cisinski-asterisque}.

We note that our development does not seem easily amendable to closing the generating trivial cofibrations $\cal{J}$ under Leibniz application with $i^1$.
The difficulties lie in two points: first, right properness of $(\TC, \F)$, and second, extension of $\F$ along $\TC$ using the equivalence extension property.
\end{remark}

\begin{remark} \label{tiny-codomains}
A sufficient condition for locality~\ref{local} of fibrations is that the generating trivial cofibrations $\cal{J}$ have tiny objects as codomains.
An object $X \in \cal{E}$ is \emph{tiny} if $\cal{E}(X, -)$ preserves colimits.

In the setting of a presheaf category, our locality condition \ref{local} corresponds precisely to the one of~\cite[Definition~3.7]{cisinski-univalence}, which requires a map to be a fibration as soon as all its pullbacks to representables are fibrations.
However, we prefer not having to refer to specific features of the underlying category such as representables.
Both conditions are satisfied if generating trivial cofibrations have representable codomain as in~\cite{shulman:univalence-simplicial}.
\end{remark}

\begin{lemma}
Consider a commuting triangle as follows:
\[
\xymatrix{
&
  Y
  \ar[dr]^{q}
\\
  Z
  \ar[ur]^{p}
  \ar[rr]_{r}
&&
  X
\rlap{.}}
\]
We have:
\begin{enumerate}[label=(\roman*)]
\item if $p$ and $q$ are trivial fibrations, then so is $r$,
\item if $p$ and $r$ are trivial fibrations, then so is $q$,
\item if $q$ and $r$ are trivial fibrations and $p$ is a fibration, then $p$ is a trivial fibration.
\end{enumerate}
\end{lemma}

\begin{proof}
Case (i) is vertical composability of trivial fibrations.
For case (ii), since $p$ is a trivial fibrations, it has a section.
This makes $q$ a retract of $r$ in the arrow category.
Since $r$ is a trivial fibration, so is $q$.
The remainder of the proof will be devoted to the main case~(iii).

The below diagram exhibits $p$ as a retract of $I \obackslash Z \to Y \times_X (I \obackslash X)$ (where we omitted drawing the horizontal composite identities):%
\footnote{Here, the pullback is taken with respect to $\delta_0 \hatobackslash X \co I \obackslash X \to X$ as indicated by the order of symbols in the pullback.
We adopt this convention for the rest of this proof.}
\[
\xymatrix@C+4em{
  Z
  \ar[r]^-{\epsilon \hatobackslash Z}
  \ar[d]^{p}
&
  I \obackslash Z
  \ar[r]^-{\delta_0 \hatobackslash Z}
  \ar[d]^{\angles{p \cc (\delta_0 \hatobackslash Z), I \hatobackslash m}}
&
  Z
  \ar[d]^{p}
\\
  Y
  \ar[r]^-{\angles{\id, (\epsilon \obackslash X) \cc q}}
&
  Y \times_X (I \obackslash X)
  \ar[r]^-{\pi_0}
&
  Y
\rlap{.}}
\]
It will thus suffice to prove the middle map a trivial fibration.
This map decomposes as follows:
\[
\xymatrix{
  I \obackslash Z
  \ar[r]
&
  (I \obackslash Y) \times_Y Z
  \ar[r]
&
  Y \times_X (I \obackslash X) \times_X Z
  \ar[r]
&
  Y \times_X (I \obackslash X)
\rlap{.}}
\]
The first map is the trivial fibration $\delta_1 \hatobackslash p$, using the assumption that $p$ is a fibration.
The second map is a base change of the trivial fibration $\bracks{\delta_0, \delta_1} \hatobackslash q$, using the assumption that $q$ is a trivial fibration.%
\footnote{In fact, going back to case~(ii), it would suffice to assume only that $\bracks{\delta_0, \delta_1} \hatobackslash q$ has a section.}
The third map is a base change of $r$, also assumed a trivial fibration.
\end{proof}

\begin{corollary} \label{trivial-fibration-2-out-of-3}
Trivial fibrations satisfy 2-out-of-3 relative to (\ie, in the subcategory of) fibrations.
\qed
\end{corollary}

\begin{lemma} \label{span-of-left-maps-over-right-map}
Let $(\cal{L}, \cal{R})$ be a wfs such that $\cal{L}$ is closed under Leibniz application of $i^1 \co \partial I \to I$.
Given a commuting triangle
\[
\xymatrix{
&
  A
  \ar[dl]_{\in \cal{L}}
  \ar[dr]^{\in \cal{L}}
\\
  Y
  \ar[rr]_{\in \cal{R}}
&&
  X
\rlap{,}}
\]
the map $Y \to X$ is a $k$-oriented costrong deformation retract for any $k \in \braces{0, 1}$.
\end{lemma}

\begin{proof}
This is standard reasoning, using two $(\cal{L}, \cal{R})$-lifting problems against the map $Y \to X$.
Lifting $A \to X$ constructs the section.
Then lifting the Leibniz application of $\partial I \to I$ to $A \to Y$ constructs the needed relative homotopy.
\end{proof}

\begin{lemma} \label{strong-codef-retract-iff-trivial}
Let $p \co Y \to X$ be a fibration.
Then there are functors in all directions between:
\begin{enumerate}[label=(\roman*)]
\item $p$ is a trivial fibration,
\item one of:
  \begin{enumerate}[label=(\alph*)]
  \item $p$ is a (left or right) strong codeformation retract,
  \item $p$ is a (left or right) strong homotopy equivalence.
  \end{enumerate}
\end{enumerate}
\end{lemma}

\begin{proof}
Fix $k \in \braces{0, 1}$, considering only $k$-oriented data in~(ii).
This is justified by~(i) being independent of $k$.
For the direction from~(ii.a) to~(ii.b), note that strong codeformation retracts are special cases of strong homotopy equivalences.
For the direction from~(ii.b) to~(i), note that $p$ being a $k$-oriented strong homotopy equivalence exhibits $p$ as a retract of $\delta_k \hatobackslash p$, a trivial fibration.
For the direction from~(i) to~(ii.a), we apply \cref{span-of-left-maps-over-right-map} with $(\cal{L}, \cal{R}) \defeq (\C, \TF)$ and $A \defeq 0$, using that every object is cofibrant.
\end{proof}

\begin{corollary} \label{span-property-holds}
The wfs's $(\C, \TF)$ and $(\TC, \F)$ satisfy the span property of \cref{span-property}.  
\end{corollary}

\begin{proof}
This is \cref{span-of-left-maps-over-right-map} with $(\cal{L}, \cal{R}) \defeq (\TC, \F)$, combined with the direction from~(ii.a) to~(i) of \cref{strong-codef-retract-iff-trivial}.
\end{proof}

\begin{lemma} \label{homotopies-groupoidal}
Homotopies between maps into a fibrant object admit finitary composition and inversion operations with the expected laws satisfied up to homotopy.
\end{lemma}

\begin{proof}
Standard, using the structure of the functorial cylinder.
\end{proof}

\begin{lemma} \label{homotopy-equivalence-2-out-of-3}
In any triangle of maps between fibrant objects commuting up to homotopy, if two of the maps are homotopy equivalences, then so is the third.
\end{lemma}

\begin{proof}
Standard, using \cref{homotopies-groupoidal}.
\end{proof}

\begin{lemma} \label{homotopy-equivalences-invariant}
For any fibrant $X \in \cal{E}$, the forgetful functor $\cal{E}_{/X} \to \cal{E}$ creates homotopy equivalences between fibrant objects.
\end{lemma}

\begin{proof}
Standard, see~\cite{shulman:inverse-diagrams}.
Note that $\cal{E}$ or its slices, when restricted to fibrant objects, form in particular a type-theoretic fibration category as considered in \ibid.
\end{proof}

\begin{lemma} \label{fibrant-and-homotopy-equivalence-is-trivially-fibrant}
Let $X \in \cal{E}$ be fibrant such that $X \to 1$ is a homotopy equivalence.
Then $X$ is trivially fibrant.
\end{lemma}

\begin{proof}
Note that $X \to 1$ is automatically also a (say, $0$-oriented) strong homotopy equivalence.
It follows that $X \to 1$ is a retract of $\delta_0 \hatobackslash \top_X$ and hence a trivial fibration.
\end{proof}

\begin{corollary} \label{fibrant-and-homotopy-equivalence-is-trivial-fibration}
A fibration between fibrant objects is a homotopy equivalence if and only if it is a trivial fibration.
\end{corollary}

\begin{proof}
For the forward direction, combine \cref{homotopy-equivalences-invariant,fibrant-and-homotopy-equivalence-is-trivially-fibrant}.
For the reverse direction, apply \cref{strong-codef-retract-iff-trivial}.
Alternatively, one may infer this from the results in~\cite{shulman:inverse-diagrams}.
\end{proof}

\subsection{Path objects}
\label{path-objects}

Consider the factorization
\[
\xymatrix@C+1em{
  \partial I
  \ar[r]^-{i^1}
&
  I
  \ar[r]^-{\epsilon}
&
  \Id
}
\]
of the codiagonal.
The first map is valued in cofibrations by \cref{boundary-inclusion}.
Its two components form sections to the section map and are valued in trivial cofibrations.
Right adjoint application to any object $X \in \cal{E}$ produces the \emph{path object factorization}
\[
\xymatrix@C+3em{
  X
  \ar[r]^-{\epsilon \obackslash X}
&
  I \obackslash X
  \ar[r]^-{\angles{\delta_0 \obackslash X, \delta_1 \obackslash X}}
&
  X \times X
}
\]
of the diagonal at $X$.
The first map is called \emph{reflexivity map}.
The second map is called \emph{boundary projection}.
Its components are called \emph{endpoint projections} and form retractions of the reflexivity map.
If $X$ is fibrant, then the boundary projection is a fibration by \ref{boundary-fibration} and the endpoint projections are trivial fibrations.

We will frequently use these notions in a relative setting, \ie in a slice category $\cal{E}_{/Y}$ for $Y \in \cal{E}$, giving rise to a path object factorization for any map $X \to Y$.
Note that the construction is functorial and stable under change of base.


\subsection{Mapping cocylinder}

Let $f \co X_0 \to X_1$ be a map between fibrant objects.
By pulling back $\delta_0 \obackslash X_1 \co I \obackslash X_1 \to X_1$ along $f$ as shown below:
\[
\xymatrix@!C{
&
  X_0
  \ar[r]^-{f}
  \ar[d]
  \ar[ddl]_{f}
  \fancypullback{[d]}{[r]}{[dr]}[0.25cm][0.5]
&
  X_1
  \ar[d]^{\epsilon \obackslash X_1}
\\&
  M f
  \ar[r]
  \ar[d]
  \fancypullback{[d]}{[r]}{[dr]}[0.25cm][0.5]
&
  I \obackslash X_1
  \ar[d]^{i^1 \obackslash X_1}
\\
  X_1
&
  X_0 \times X_1
  \ar[r]
  \ar[d]
  \ar[l]^{\pi_1}
  \fancypullback{[d]}{[r]}{[dr]}[0.25cm][0.5]
&
  X_1 \times X_1
  \ar[d]^{\pi_0}
\\&
  X_0
  \ar[r]_-{f}
&
  X_1
\rlap{,}}
\]
we construct the \emph{mapping cocylinder factorization}
\[
\xymatrix{
  X_0
  \ar[r]^-{j}
&
  M f
  \ar[r]^-{e}
&
  X_1
}
\]
of $f$ where the second map is the composition of $M f \to X_0 \times X_1$ followed by $\pi_1$.

Intuitively, the fibers of $M f \to X_1$ are the homotopy fibers of $f \co X_0 \to X_1$.
As a base change of $I \obackslash X_1 \to X_1$, note that $f^* d_0 \co M f \to X_0$ inherits the structure of a strong codeformation retract with section $j$.
As a composition of $M f \to X_0 \times X_1$, which is a base change of $i^1 \obackslash X_1$, and $\pi_1 \co X_0 \times X_1 \to X_1$, which is a base change of $X_0 \to 1$, note that $e \co M f \to X_1$ is a fibration.
As with the path object, the mapping cocylinder factorization is functorial and stable under change of base.

\subsection{Equivalences}

Continuing the setting of the previous subsection, we call $f \co X_0 \to X_1$ an \emph{equivalence} if the second map $M f \to X_1$ of its mapping cocylinder factorization is a trivial fibration.

\begin{lemma} \label{fibration-equivalence-same-as-trivial-fibration}
A fibration $f \co X_0 \to X_1$ between fibrant objects is an equivalence if and only if it is a trivial fibration.
\end{lemma}

\begin{proof}
The claim follows, for example, by looking at the square
\[
\xymatrix@C+1em{
  I \obackslash X_0
  \ar@{->>}[r]_-{\triv}^-{\delta_1 \obackslash X_1}
  \ar@{->>}[d]^{\triv}_{\delta_0 \hatobackslash f}
&
  X_0
  \ar@{->>}[d]^{f}
\\
  M f
  \ar@{->>}[r]_{e}
&
  X_1
\rlap{.}}
\]
The top map is a trivial fibration as $X_1$ is fibrant.
The left map is a trivial fibration as $f$ is a fibration.
By \cref{trivial-fibration-2-out-of-3}, the bottom map is a trivial fibration exactly if the right map is.
\end{proof}

\begin{lemma} \label{equivalence-variations}
For a map $f \co X_0 \to X_1$ between fibrant objects, the following are equivalent:
\begin{enumerate}[label=(\roman*)]
\item $f$ is an equivalence,
\item $f$ decomposes as a section of a trivial fibration followed by a trivial fibration,
\item $f$ is a homotopy equivalence.
\end{enumerate}
\end{lemma}

\begin{proof}
The direction from~(i) to~(ii) is immediate.
The rest follow from repeated applications of \cref{fibrant-and-homotopy-equivalence-is-trivial-fibration,homotopy-equivalence-2-out-of-3}.
\end{proof}

\section{The equivalence extension property}
\label{section:glueing}

We continue working in the setting established in \cref{section:setting,section:preliminary}.
We will now give the central technical aspect of our development, the proof of the equivalence extension property.
It derives from the glueing construction of~\cite{cohen-et-al:cubicaltt} in their cubical sets, but is generalized to an abstract setting and presented in categorical terms.

\begin{proposition}[Equivalence extension property] \label{glueing-core}
Consider the solid part of the diagram
\begin{equation} \label{glueing-core:0}
\begin{gathered}
\xymatrix{
  X_0
  \ar@{.>}[rr]
  \ar[dr]
  \ar@{->>}[dd]
  \fancypullback{[dd]}{[rrr]}{[ddrrr]}[0.3cm][0.6]
&&
  Y_0
  \ar@{.>}[dr]
  \ar@{.>>}[dd]|{\hole}
&\\&
  X_1
  \ar[rr]
  \ar@{->>}[dl]
  \fancypullback{[dl]}{[rr]}{[dr]}[0.3cm][0.6]
&&
  Y_1
  \ar@{->>}[dl]
\\
  A
  \ar[rr]
&&
  B
&
}
\end{gathered}
\end{equation}
where the lower square is a pullback, the maps $X_0 \to A$ and $Y_1 \to B$ are fibrations.
Assume that:
\begin{enumerate}[label=(\roman*)]
\item \label{glueing-core:equiv} the map $X_0 \to X_1$ is an equivalence over $A$,
\item \label{glueing-core:cofib} the map $A \to B$ is a cofibration.
\end{enumerate}
Then there is $Y_0$ fitting into the diagram as indicated such that the back square is a pullback, the map $Y_0 \to B$ is a fibration, and the map $Y_0 \to Y_1$ is an equivalence over $B$.
\end{proposition}

\begin{proof}
By assumption~\ref{glueing-core:cofib} and stability of cofibrations under pullback, the map $X_1 \to Y_1$ is also a cofibration.
In particular, base change and pushforward along $X_1 \to Y_1$ form a reflection.
Let $X_0 \to M \to X_1$ be the mapping cocylinder factorization of $f$ over $A$.
We let $Y_0 \to N \to Y_1$ be its pushforward along $X_1 \to Y_1$, defining all dotted maps in~\eqref{glueing-core:0} in the process.
Note that $M \to X_1$ is a trivial fibration by assumption~\ref{glueing-core:equiv}.
By stability of trivial fibrations under pushforward, we have that $N \to Y_1$ is a trivial fibration.

Of central importance will be the adjunction $A^* \dashv A_*$ relative to $B$.
It induces the monad on $\cal{E}_{/B}$ of exponentiation with $A$, denoted $(-)^A \defeq A_* A^*$ with unit $\eta$.
Observe, for example by normalizing polynomial functors, that pushforward along $X_1 \to Y_1$ can be written as the composition
\begin{equation} \label{glueing-core:1}
\begin{gathered}
\xymatrix{
  \cal{E}_{/X_1}
  \ar[r]^-{A_*}
&
  \cal{E}_{/Y_1^A}
  \ar[r]^-{\eta_{Y_1}^*}
&
  \cal{E}_{/Y_1}
\rlap{.}}
\end{gathered}
\end{equation}

For the remainder of the proof, we will look at the map
\begin{equation} \label{glueing-core:2}
\begin{gathered}
\xymatrix@C+6em{
  Y_1
  \ar[r]^-{\angles{(\epsilon \obackslash_B Y_1^A) \cc \eta_{Y_1}, \id}}
&
  (I \obackslash_B Y_1)^A \times_{Y_1^A} Y_1
}
\end{gathered}
\end{equation}
living in $\cal{E}_{/Y_1^A}$ via $(\delta_0 \obackslash Y_1)^A \cc \pi_0$.
We make two claims:
\begin{enumerate}[label=(\alph*)]
\item \label{glueing-core:base-change}
the map $Y_0 \to N$ arises as a base change of it (along the map $A_* X_0 \to Y_1^A$),
\item \label{glueing-core:factorization}
it factors as a section (over $Y_1^A$) of a trivial fibration followed by a trivial fibration.
\end{enumerate}
By stability under pullback, the map $Y_0 \to N$ will then inherit a factorization into a section (over $A_* X_0$ and hence also $B$) of a trivial fibration followed by a trivial fibration.
Composing with the trivial fibration $N \to Y_1$ and the fibration $Y_1 \to B$, this will exhibit $Y_0$ as a retract of a fibrant object over $B$ and simultaneously $Y_0 \to Y_1$ as an equivalence by \cref{equivalence-variations}.%
\footnote{Actually, the induced decomposition of $Y_0 \to Y_1$ into a section of a trivial fibration followed by a trivial fibration turns out to be its mapping cocylinder factorization, so the use of \cref{equivalence-variations} to establish it as an equivalence is not needed.}

For claim~\ref{glueing-core:base-change}, recall the construction of the mapping cocylinder factorization: the map $X_0 \to M$ over $X_1$ is a pullback of
\[
\xymatrix@C+2em{
  X_1
  \ar[r]^-{\epsilon \obackslash_A X_1}
&
  I \obackslash_A X_1
}
\]
along $X_0 \times_A X_1 \to X_1 \times_A X_1$ living in $\cal{E}_{/X_1}$ via the second projection.
Note that $\epsilon \obackslash_A X_1$ is itself the pullback of $\epsilon \obackslash_B Y_1$ along $A \to B$ (see \cref{cylinder-slicing}).
Using the description~\eqref{glueing-core:1} of pushforward along $X_1 \to Y_1$ and preservation of pullbacks by right adjoints, it follows that $Y_0 \to N$ is a pullback of the map~\eqref{glueing-core:2} along $A_* X_0 \times_B Y_1 \to Y_1^A \times_B Y_1$ in $\cal{E}_{/Y_1}$, hence also along $A_* X_0 \to Y_1^A$ in $\cal{E}_{/B}$.

For claim~\ref{glueing-core:factorization}, we factorize as follows:
\[
\xymatrix@C+0em{
  Y_1
  \ar[rr]^-{\angles{(\epsilon \obackslash_B Y_1^A) \cc \eta_{Y_1}, \id}}
  \ar[dr]_{\epsilon \obackslash_B Y_1}
&&
  (I \obackslash_B Y_1)^A \times_{Y_1^A} Y_1
\rlap{.}\\&
  I \obackslash_B Y_1
  \ar[ur]_{A \hatobackslash_B \delta_1 \obackslash_B Y_1}
}
\]
The first factor is a section of the map $\delta_0 \obackslash_B Y_1$ (over $Y_1^A$), a trivial fibration since $Y_1$ is fibrant over $B$.
The second factor is the Leibniz exponential of the trivial fibration $\delta_1 \obackslash_B Y_1$ with the cofibration $A \to B$ (assumption~\ref{glueing-core:cofib}), a trivial fibration by~\cref{cofibration-exp-trivial-fibrations}.
\end{proof}

\begin{remark}
The decomposition strategy in the above proof, factoring the given map $X_0 \to X_1$ via the mapping cocylinder factorization into a specific strong deformation retract followed by a trivial fibration, is evocative of the related proof of univalence in the simplicial setting of~\cite[Theorem~3.4.1]{voevodsky-simplicial-model}, which would factor the map $X_0 \to X_1$ as a cofibration that is a strong deformation retract followed by a trivial fibration (note that \ibid takes the Kan model structure on simplicial sets for granted).
The difference is that the mapping cocylinder factorization does not in general produce a cofibration as its first factor.

Instead, in order to proceed similarly to~\cite[Theorem~3.4.1]{voevodsky-simplicial-model}, one could use the (cofibration, trivial factorization)-factorization of $X_0 \to X_1$, apply \cref{homotopy-equivalence-2-out-of-3} to make the cofibration into a homotopy equivalence, and then show that a cofibration between fibrant objects that is a homotopy equivalence is also trivial cofibration and hence a strong deformation retract (relative to $A$).
\end{remark}

\begin{remark} \label{algebraic}
Note that~\cite{cohen-et-al:cubicaltt} uses an algebraic (or uniform) notion of fibration where chosen lifts against generating trivial cofibrations are part of the data of a fibration; see~\cite{gambino-sattler:frobenius} for an abstract treatment.
In that context, for showing the algebraic analogue of the extension property of fibrations along trivial cofibrations, it is required that the back pullback square in \eqref{glueing-core:0} additionally forms a morphism of fibrations, \ie cohering with the chosen lifts of $X_0 \fib A$ and $Y_0 \fib B$.

To accomplish this, one needs to additionally assume that cofibrations are closed under right adjoint application of the functorial cylinder $I$, complementing condition~(v) of \cref{suitable-wfs}.
This is an equivalent phrasing of the $\forall$-condition of~\cite{cohen-et-al:cubicaltt}, requiring that the right adjoint to pullback of subobjects along any component of the contraction $\epsilon \co \Id \to I$ preserves cofibrations.
With this, it is possible in any pullback square $p' \to p$ of uniform fibrations, not necessarily cohering with the lifting structures, to replace the lifting structure on $p$ by one that makes $p' \to p$ into a morphism of fibrations.

However, in our setting of ordinary (non-uniform) fibration, this assumption is not needed.
We leave the treatment of the algebraic case, dealing with algebraic wfs's and algebraic model structures and extending the treatment in the second part of~\cite{gambino-sattler:frobenius}, for further work.
\end{remark}

We also give a lemma whose use will be closely related to the equivalence extension property.
It will be needed in \cref{section:composition}.

\begin{lemma} \label{path-to-homotopy-equivalence}
Any fibration $p \co X \to I \otimes A$ gives rise to a homotopy equivalence over $A$ between the fibers $X_0$ and $X_1$, where $p_k \co X_k \to A$ is the pullback of $X$ along $\delta_k \otimes A$ for $k \in \braces{0, 1}$.
\end{lemma}

\begin{proof}
This is standard, using the structure of the functorial cylinder and that every object is cofibrant.

Lifting $\delta_0 \otimes X_0$ against $p$ gives a map $f_0 \co I \otimes X_0 \to X$.
Precomposing $f_0$ with $\delta_1 \otimes X_0$ induces a map $u_0 \co X_0 \to X_1$.
We define maps $f_1$ and $u_1$ dually.
To see that $u_1 u_0 \sim \id_{X_0}$, we use a lifting problem
\[
\xymatrix{
  I \otimes X_0 +_{X_0} I \otimes X_0 +_{X_0} I \otimes X_0
  \ar[r]
  \ar[d]_{\delta_1 \hatotimes i^1 \otimes X_0}
&
  X
  \ar[d]
\\
  I \otimes I \otimes X_0
  \ar[r]_{I \otimes \epsilon \otimes p_0}
  \ar@{.>}[ur]
&
  I \otimes A
}
\]
where the three components of the top map are given by $f_0$, $u_0 \cc (\epsilon \otimes X_0)$, and $f_1 \cc (I \otimes u_0)$.
Precomposing with $\delta_0 \otimes I \otimes X_0$ induces the required homotopy.
We see $u_0 u_1 \sim \id_{X_1}$ dually.
\end{proof}

Note that the proof of \cref{path-to-homotopy-equivalence} does not make use of connections.

\section{Composition versus Filling}
\label{section:composition}

In order to give a categorical treatment of \emph{composition} as introduced by~\cite{cohen-et-al:cubicaltt} in comparison to lifting against $\cal{J}$ termed \emph{filling} in \ibid, we will generalize the lifting relation from arrows to squares.
We write $\cal{E}^\Box$ for the category of commuting squares in $\cal{E}$, defined as the arrow category of $\cal{E}^\to$.
We will write an object of $\cal{E}^{\Box}$ as $(u, v) \co f \to g$ where $u \co \dom(f) \to \dom(g)$ and $v \co \cod(f) \to \cod(g)$.

\begin{definition} \label{square-lifting}
We say a square $(u, v) \co l' \to l$ \emph{lifts} against an arrow $r$ if for dashed maps making the diagram
\begin{equation} \label{square-lifting:0}
\begin{gathered}
\xymatrix{
  \bullet
  \ar[r]^{u}
  \ar[d]_{l'}
&
  \bullet
  \ar@{-->}[r]
  \ar[d]_(0.3){l}
&
  \bullet
  \ar[d]^{r}
\\
  \bullet
  \ar[r]_{v}
  \ar@{.>}[urr]
&
  \bullet
  \ar@{-->}[r]
&
  \bullet
}
\end{gathered}
\end{equation}
commute, there is a dotted filler as indicated.
\end{definition}

We allow ourselves to view any arrow $l$ as a square via the identity $\id \co l \to l$.
Observe that $l$ lifting against an arrow $r$ does not depend on whether we see $l$ as an arrow or a square.

The Galois connection $\liftl{(-)} \dashv \liftr{(-)}$ between classes of arrows generalizes to a Galois connection between classes of squares on the left and classes of arrows on the right, denoted using the same operators.
In fact, the former adjunction factors through the latter via the adjunction generated by the inclusion of arrows into squares described above.

\begin{remark}
It is possible to generalize \cref{square-lifting} further to liftings of squares on the left against squares on the right.
This comes with an analogous Galois connections between lifting operators that the Galois connections considered previously factor through.
Although this makes the situation more symmetric, we do not have any need for that generality here.
\end{remark}

\begin{remark}
In a category with pushouts, the lifting problem~\eqref{square-lifting:0} is equivalent to a lifting problem of $(\id, v) \co l'' \to l$ against $r$ where $l''$ is the pushout of $l'$ along $v$.
That is, the lifting relation can be reduced to squares with an identity as top map.
Even though our setting has pushouts, we prefer to work with arbitrary squares: first, because the square~\eqref{functorial-cylinder:disjoint-endpoints} we will be working with naturally arises with a non-identity at the top; second, because the analogous reduction of the extension relation of \cref{extension-along-squares} depends on a van Kampen condition of the pushout defining $l''$ and a locality assumption on the class $\cal{B}$.
\end{remark}

\begin{definition} \label{biased-retract}
A \emph{biased retract} of a square $g' \to g$ to a square $f' \to f$ consists of maps $f' \to g'$ and $g \to f$ such that the following diagram commutes:
\begin{equation} \label{biased-retract:0}
\begin{gathered}
\xymatrix{
  f'
  \ar[r]
  \ar@{.>}[d]
  \ar@/_2em/[dd]_{\id}
&
  f
  \ar@/^2em/[dd]^{\id}
\\
  g'
  \ar[r]
&
  g
  \ar@{.>}[d]
\\
  f'
  \ar[r]
&
  f
}
\end{gathered}
\end{equation}
\end{definition}

Note that a retract between arrows gives rise to a biased retract between the induced squares.

\begin{lemma} \label{lifting-closed-under-biased-retract}
Given a class of maps $\cal{R}$, the class of squares that $\cal{R}$ lifts against is closed under biased retract.
\end{lemma}

\begin{proof}
Straightforward diagram chasing.
\end{proof}

Given a map of arrows $f' \to f$ that is a section, note that $f'$ (seen as a square) is a biased retract of $f' \to f$.

We recall the square $\theta_k \co \bot_{\Id} \to \delta_k$ for $k \in \braces{0, 1}$ from~\cite{gambino-sattler:frobenius}, which is just~\eqref{functorial-cylinder:disjoint-endpoints} read in either horizontal or vertical direction.
We let $J'$ denote the set of Leibniz applications of $\theta_0$ and $\theta_1$ to generating cofibrations
\begin{equation} \label{gen-triv-cof-comp}
\cal{J}' \defeq \set{\theta_k \hatotimes m \co m \to \delta_k \hatotimes m }{k \in \braces{0, 1}, m \in \cal{I}}
\rlap{.}
\end{equation}
Observe that the square $\theta_k \hatotimes m$ in $\cal{J}'$ is a biased retract of the arrow $\delta_k \hatotimes m$ in $\cal{J}$.
As shown in \cite[Lemma~3.4]{gambino-sattler:frobenius}, the presence of connections makes $\theta_k \hatotimes \delta_k \hatotimes m$ a section for any $m \in \cal{I}$.
Hence, conversely, the arrow $\delta_k \hatotimes m$ is a biased retract of the square $\theta_k \hatotimes \delta_k \hatotimes m$.
Since $\cal{I}$ is a generator for $\C$, an easy adjointness argument shows that $\cal{J}'$ has the same right lifting class of arrows as the class obtained by replacing $\cal{I}$ with $\C$ in~\eqref{gen-triv-cof-comp}.
In view of \cref{lifting-closed-under-biased-retract}, we thus have shown the following.

\begin{corollary}[Filling and Composition are equivalent] \label{composition-is-filling}
We have $\F = \liftr{\cal{J}} = \liftr{(\cal{J}')}$.
\qed
\end{corollary}

\section{The extension property}
\label{section:extension}

\begin{definition}[Extension along squares] \label{extension-along-squares}
A class of maps $\cal{B}$ has \emph{extension along a square} $l' \to l$,
\[
\xymatrix{
  U'
  \ar[r]
  \ar[d]_{l'}
&
  U
  \ar[d]^{l}
\\
  V'
  \ar[r]
&
  V
\rlap{,}}
\]
if for every map in $\cal{B}$ into $U$, there is a map in $\cal{B}$ into $V'$ that pulls back to the same map into $U'$.
\end{definition}

Given a map $l$, note that $\cal{B}$ extends along $l$ precisely if $\cal{B}$ extends along it when seen as a square $l \to l$.

\begin{lemma} \label{extension-closed-under-biased-retract}
Given a class of maps $\cal{B}$, the class of squares that $\cal{B}$ extends along is closed under biased retract.
\end{lemma}

\begin{proof}
Essentially following the structure of the proof of \cref{lifting-closed-under-biased-retract} (and corresponding to it in the presence of a classifier for $\cal{B}$; see \cref{classifier}).
For completeness, we still give the proof.

We work with the diagram \eqref{biased-retract:0}.
Given a $\cal{B}$-map into $\dom(f)$, we pull it back to a $\cal{B}$-map into $\dom(g)$.
Extension gives a $\cal{B}$-map into $\cod(g')$ coherent with respect to pulling back to $\dom(g')$.
Pulling back further gives a $\cal{B}$-map into $\cod(f')$ coherent with respect to pulling back to $\dom(f')$.
\end{proof}

\begin{lemma} \label{fibrations-extend-squares}
Fibrations extend along squares $\theta_k \hatotimes m \co m \to \delta_k \hatotimes m$ for $k \in \braces{0, 1}$ and $m \in \C$.
\end{lemma}

\begin{proof}
We only deal with the case $k = 1$, the case $k = 0$ is dual (observe that the notion of homotopy equivalence is symmetric).
Let $m \co A \to B$ be a cofibration.
We recall the square $\theta_1 \hatotimes m$:
\[
\xymatrix@C+2em{
  A
  \ar[r]^-{\iota_1 \cc (\delta_0 \otimes A)}
  \ar[d]_{m}
&
  B +_A I \otimes A
  \ar[d]^{\delta_1 \hatotimes m}
\\
  B
  \ar[r]_-{\delta_0 \otimes B}
&
  I \otimes B
\rlap{.}}
\]
Given a fibrant object over the top right corner, pulling back to the components of the coproduct and using \cref{path-to-homotopy-equivalence}, we obtain precisely the input data for \cref{glueing-core} as in the diagram shown there.
The resulting fibration $Y_0 \to B$ is the needed extension.
\end{proof}

Combined with the discussion preceding \cref{composition-is-filling} amd \cref{extension-closed-under-biased-retract}, we finally obtain the following.

\begin{corollary} \label{fibrations-extension-gen-triv-cof}
Fibrations extend along the generating trivial cofibrations $\cal{J}$ from~\eqref{gen-triv-cof}.
\qed
\end{corollary}

It remains to go from extension of fibrations along generating trivial cofibrations to extension of fibrations along arbitrary trivial cofibrations.
From Quillen's small object argument~\cite{quillen-homotopical}, we have an explicit description of the latter in terms of the former: any trivial cofibration is a retract of a transfinite composition of cobase changes of coproducts of generating trivial cofibrations (we include coproducts here to avoid a use of the axiom of choice).

\begin{lemma} \label{extension-closure}
The class of cofibrations that fibrations extend along is closed under:
\begin{enumerate}[label=(\roman*)]
\item coproducts,
\item cobase change,
\item transfinite compositions,
\item retracts.
\end{enumerate}
\end{lemma}

\begin{proof}
The last part was already proven in more general form in \cref{extension-closed-under-biased-retract}.

The remaining claims make use of the following pattern.
Given a van Kampen colimit in $\cal{E}$, we may extend any cartesian diagram $F$ in $\cal{E}^\to$ whose codomain part is the given colimit diagram minus the tip towards the tip by taking the colimit in $\cal{E}^\to$, yielding a cartesian colimit diagram.
Given fibration structures on all objects of $F$, locality of fibrations as assumed in~\ref{local} implies that $\colim F$ is again a fibration.

For part~(i), given a fibration into the domain of a coproduct of arrows, we pull it back to the domains of its components.
We then extend and take the colimit.
Note that the domain and codomain parts of the coproduct are van Kampen by extensivity as assumed with the suitable category $\cal{E}$.
It follows that the resulting fibration pulls back to the starting fibration.

For part~(ii), van Kampenness of the relevant pushouts follows from adhesivity in condition~(ii) of \cref{suitable-wfs}.
Given a fibration into the domain of the cobase change, we pull it back, extend it, and then take the colimit of arrows.

For part~(iii), van Kampenness of the transfinite composition and its initial segments follows from exhaustiveness in condition~(ii) of \cref{suitable-wfs}.
Given a fibration into the start of the chain of cofibrations, we recursively extend it, using extension for the successor case and van Kampen plus locality in the limit case.
\end{proof}

\begin{remark} \label{classifier}
\newcommand{\M}{\cal{B}}
\newcommand{\CC}{\mathbb{C}}
\newcommand{\Up}{\widetilde{U}}

In the presence of a classifier $U$ for a class of maps $\M$, closure of the class of cofibrations that $\M$ extends along under weak saturation as in \cref{extension-closure} may be inferred more directly.
As detailed in~\cite{cisinski-univalence}, extension of $\M$ along a cofibration $m$ is equivalent to lifting of $m$ against $U \to 1$ provided we have a map $\pi \co \Up \to U$ with the following properties:
\begin{enumerate}[label=(\roman*)]
\item
the elements of $\M$ are those maps arising as a pullback of $\pi \co \Up \to U$,
\item
given pullback squares $(u, v) \co p \to p'$ and $(x, y) \co p \to \pi$ with $p' \in \M$ (and hence $p \in \M)$ and $v$ a cofibration, there is a map $y' \co Y' \to U$ such that $y' v = y$.
\end{enumerate}
It remains to note that both cofibrations and the class of maps lifting against $U \to 1$ are closed under weak saturation.

If $\cal{E}$ is presheaves over a small category $\CC$, the subclass of locally $\kappa$-small fibrations (\ie whose fibers have cardinality below $\kappa$) admits such a classifier $U_{<\kappa}$ for any cardinal $\kappa$ greater than the cardinality of arrows of $\CC$ as shown in~\cite{cisinski-univalence} and using different methods in~\cite{shulman:univalence-simplicial}.
However, although not detailed in~\cite{cisinski-univalence}, if working constructively, the proof of property~(ii) needs the additional assumption that cofibrations are valued in natural transformations with decidable monomorphisms as components.
Assuming arbitrarily large regular cardinals, this may be used to derive an easier proof of \cref{extension-closure}, corresponding to fibrancy of $U_{<\kappa}$, and subsequently of \cref{fibrations-extension-triv-cof} since the construction of \cref{glueing-core} and the derived \cref{fibrations-extension-gen-triv-cof} preserve locally $\kappa$-small fibrations.

One reason for our choice of proof even in the presheaf setting, working explictly with van Kampen colimits, is the elimination of the need for arbitrarily large cardinals and also (when working constructively) decidability of the components of cofibrations in case of the approach of~\cite{cisinski-univalence} (this is avoided in the approach of~\cite{shulman:univalence-simplicial}).
Note that in a predicative setting, this need for restricting to such cofibrations may still be necessary for ensuring the wfs $(\C, \TF)$ is cofibrantly generated.
\end{remark}

\begin{corollary} \label{fibrations-extension-triv-cof}
Fibrations extend along trivial cofibrations.
\end{corollary}

\begin{proof}
Combine \cref{fibrations-extension-gen-triv-cof} with \cref{extension-closure}.
\end{proof}

\section{The model structure}
\label{section:result}

We finally proved everything needed to apply the sufficient criterion for a right proper model structure developed in \cref{section:criterion}.
To make the statement more self-contained, we include a reminder of our suitable setting.

\begin{theorem}[Main theorem] \label{main-theorem}
Let $\cal{E}$ be a suitable category as per \cref{subsection:category}: it is locally presentable, locally cartesian closed, and infinitary extensive.

Let $I \otimes (-)$ be a suitable functorial cylinder as per \cref{subsection:cylinder}: it is left adjoint and has contractions, disjoint endpoint inclusions $\delta_0, \delta_1 \co \Id \to I$, and connections.

Let $(\C, \TF)$ be a suitable wfs as per \cref{subsection:wfs}: it is cofibrantly generated and its left maps are adhesive; exhaustive; closed under pullback, finitary union, application of $I$; and contain the components of the endpoint inclusions of $I$.

Define
\[
\F \defeq \liftr{(\braces{\delta_0, \delta_1} \hatotimes \C)} = \liftr{\set{\delta_k \hatotimes m}{k \in \braces{0, 1}, m \in \C}}
\]
and $\TC \defeq \liftl{\F}$ as per \cref{subsection:fibrations}.
Assume that $\F$ is closed under $[\delta_0, \delta_1] \hatobackslash (-)$ \ref{boundary-fibration} and local \ref{local}.
Then the category $\cal{E}$ forms a right proper combinatorial model category with wfs's $(\C, \TF)$ and $(\TC, \F)$.
\end{theorem}

\begin{proof}
We apply \cref{right-proper-model-structure}:
\begin{itemize}
\item condition~(i) holds by \cref{span-property-holds},
\item condition~(ii) holds by \cref{trivial-fibration-2-out-of-3},
\item condition~(iii) holds by \cref{fibrations-extension-triv-cof,triv-fib-extend-along-cof},
\item condition~(iv) holds by \cref{frobenius}.
\qedhere
\end{itemize}
The resulting model structure is combinatorial since $\cal{E}$ is locally presentable and cofibrantly generated by condition~(i) of \cref{suitable-wfs} and the setup in \cref{subsection:fibrations}.
\end{proof}

In choosing our setting, recall that we required the exactness conditions of adhesivity and exhaustiveness only for cofibrations (see condition~(ii) of \cref{suitable-wfs}).
This is to include situations in which not all monomorphisms are \eg adhesive, but where the cofibrations can still be chosen as a proper subclass of all monomorphisms.

Let us now simplify our the setting.
We specialize to the case where the cofibrations are precisely the monomorphisms.
For convenience, we also assume that $\cal{E}$ is an exact category in the sense of Barr (note that $\cal{E}$ is already regular as it is locally cartesian closed and has coequalizers).
By a variant of Giraud's Theorem \cite[Theorem~C2.2.8, condition~(vii)]{johnstone:elephant}, our assumptions on $\cal{E}$ are then equivalent to $\cal{E}$ being a Grothendieck topos.

Since $\cal{E}$ is a Grothendieck topos, the class $\cal{M}$ of monomorphisms is weakly saturated, \ie $(\cal{M}, \liftr{\cal{M}})$ is a wfs.
Following Cisinski~\cite{cisinski-asterisque}, a \emph{cellular model} of $\cal{E}$ is a small set of maps $I$ such that $\cal{M} = \liftl{(\liftr{I})}$, \ie a witness that $(\cal{M}, \liftr{\cal{M}})$ is cofibrantly generated.
Recall that a Cisinski model structure is a cofibrantly generated model structure on a Grothendieck topos with cofibrations the monomorphisms.

Simplifying also the functorial cylinder, we assume that it comes from tensoring with an interval object with respect to a closed symmetric monoidal structure on $\cal{E}$ as explained in the background section of~\cite{gambino-sattler:frobenius}.
The structure and axioms for a suitable interval object mirror those of a suitable functorial cylinder.
As per \cref{swap-arguments}, this ensures assumption \ref{boundary-fibration} holds.

\begin{corollary} \label{main-corollary}
Let $(\cal{E}, \top, \otimes)$ be a closed symmetric monoidal Grothendieck topos admitting a cellular model.
Let $I$ be an interval object in $(\cal{E}, \top, \otimes)$ with contraction, disjoint endpoint inclusions $\delta_0, \delta_1 \co \top \to I$, and connections such that $I \otimes (-)$ preserves $\cal{M}$.
Define $\F \defeq \liftr{(\braces{\delta_0, \delta_1} \hatotimes \cal{M})}$.
If $\F$ is local \ref{local}, then $\cal{E}$ forms a proper Cisinski model category with fibrations $\F$.

The locality condition~\ref{local} holds in particular if $\cal{E}$ has a cellular model with tiny codomains and $I \otimes (-)$ preserves tiny objects.
\end{corollary}

\begin{proof}
For the first part, we instantiate \cref{main-theorem} with $\C \defeq \cal{M}$ and the suitable functorial cylinder given by tensoring with $I$.
Note that $\delta_0$ and $\delta_1$ are valued in monomorphisms since contractions provide a retraction.

For the second part, the assumptions combine to imply that the codomains of the generators $\cal{J}$ in~\eqref{gen-triv-cof} of $\TC$ are tiny.
Hence, the class of fibrations is local by \cref{tiny-codomains}.
\end{proof}

\begin{remark} \label{main-presheaves}
Let $\cal{E}$ be a presheaf category.
Then $\cal{E}$ admits a cellular model by \cite[Prop.~1.2.27]{cisinski-asterisque}.
The tiny objects are precisely the retracts of representables.
The second part of \cref{main-corollary} can then be stated slightly weaker by requiring $\cal{E}$ having a cellular model with representable codomains and $I \otimes (-)$ preserving representables.
This has a neat reformulation as $I \obackslash (-)$ admitting a further right adjoint $R$:
\[
\newcommand{\adjoint}{\quad\dashv\quad}
I \otimes (-) \adjoint I \obackslash (-) \adjoint R
.\]
This gives a more synthetic way of handling the issue of locality of fibrations, applicable to settings that are not locally presentable, or even cocomplete.
\end{remark}

For the next level of instantiation, recall the notion of an elegant Reedy category from~\cite{bergner-rezk-elegant}.

\begin{corollary} \label{example-elegant-reedy-presheaves}
Let $(\cal{E}, \top, \otimes)$ be a closed symmetric monoidal presheaf category on an elegant Reedy category $\mathbb{C}$.
Let $I$ be an interval object in $(\cal{E}, \top, \otimes)$ with contraction, disjoint endpoint inclusions $\delta_0, \delta_1 \co \top \to I$, and connections such that $I \otimes (-)$ preserves representables and $\cal{M}$.
Then $\cal{E}$ forms a proper Cisinski model category with fibrations $\F \defeq \liftr{(\braces{\delta_0, \delta_1} \hatotimes \cal{M})}$.
\end{corollary}

\begin{proof}
We apply \cref{main-corollary}.
Presheaf categories are particular Grothendieck toposes.
Since the Reedy category $\mathbb{C}$ is elegant, the monomorphisms $\cal{M}$ are the weak saturation of the small set $\cal{I}$ of latching object inclusions of representables.
With $I \otimes (-)$ preserving representables, this takes care of locality~\ref{local} by \cref{main-presheaves}.
\end{proof}

In simplical sets, the functorial cylinder $\Delta^1 \times (-)$ does not preserve representables, forbidding a direct application of \cref{example-elegant-reedy-presheaves}.
However, we can still salvage the situation.

\begin{corollary}[Kan model structure] \label{kan-model-structure}
Simplicial sets form a proper Cisinski model structure with fibrations the usual Kan fibrations.
\end{corollary}

\begin{proof}
We apply \cref{main-corollary} as in \cref{example-elegant-reedy-presheaves} (recall that the simplex category $\Delta$ is an elegant Reedy category), except that we argue separately for locality~\ref{local} of fibrations.
The monoidal structure is the cartesian one and the interval object is $\Delta^1$.
Taking the boundary inclusions $\partial \Delta^n \to \Delta^n$ for $n \geq 0$ as generating cofibrations, the generating trivial cofibrations $\cal{J}$ defined in~\eqref{gen-triv-cof} have codomains of the form $\Delta^1 \times \Delta^n$, which for $n \geq 1$ are certainly not retracts of representables and hence not tiny.
However, as shown in~\cite[Chapter~IV, Section~2]{gabriel-zisman:calculus-of-fractions} using simple combinatorics, our class of fibrations coincides with that of the usual Kan fibrations, which are local since they are defined with respect to lifts against horn inclusions, with representables as codomains.
Again following \cref{tiny-codomains}, this shows that condition~\ref{local} holds nonetheless.
\end{proof}

\bibliographystyle{alpha}
\bibliography{}

\end{document}

%% file: uniform-kan-prelude.tex
\usepackage{amssymb,amsmath,amsthm}
\usepackage{mathtools}
\usepackage{stmaryrd}
\usepackage{enumerate}
\usepackage[all,cmtip,2cell]{xy} 
\usepackage{fouridx}
\usepackage{bbm}
\usepackage{scalerel}

\usepackage[utf8]{inputenc}
\usepackage{geometry}
\usepackage{verbatim}
\usepackage{enumitem}
\usepackage{xparse}
\usepackage{xspace}
\usepackage[draft=false]{hyperref}
\usepackage{cleveref}
\usepackage{xcolor}

\def\defthm#1#2#3#4{
  \newtheorem{#1}[theorem]{#3}
  \newtheorem*{#1*}{#3}
  \newtheorem{#2}[theorem]{#4}
  \newtheorem*{#2*}{#4}
  \crefname{#1}{#3}{#4}
  \crefname{#2}{#4}{#4}  
}

\makeatletter
\def\lam#1{{\lambda}\@lamarg#1:\@endlamarg\@ifnextchar\bgroup{.\,\lam}{.\,}}
\def\@lamarg#1:#2\@endlamarg{\if\relax\detokenize{#2}\relax #1\else\@lamvar{\@lameatcolon#2},#1\@endlamvar\fi}
\def\@lamvar#1,#2\@endlamvar{(#2\,{:}\,#1)}
\def\@lameatcolon#1:{#1}

\def\lamu#1{{\lambda}\@lamuarg#1:\@endlamuarg\@ifnextchar\bgroup{.\,\lamu}{.\,}}
\def\@lamuarg#1:#2\@endlamuarg{#1}
\makeatother

\NewDocumentCommand\fancypullbackcore{m m}{
  \POS{
    "c00";"c01":"c10"::
    "c00"+/v(1,0){#1}/="a",
    "c00"+/v(0,1){#1}/="b",
    "c00";"c11":
    "a";p+"c11"-"c01";x="sa",
    "b";p+"c11"-"c10";x="sb",
    "sa"+"sb"-"c00"="t",
    "c00";"c01":"c10"::
    "t";p+"c01"-"c11";x="x",
    "t";p+"c10"-"c11";y="y",
    "t";"x":"y"::
    (0,0);({#2},0)**\dir{-},
    (0,0);(0,{#2})**\dir{-}}}

\NewDocumentCommand\fancypullback{m m g o o}{
  \save{
    p="c00",
    {#1}="c01",
    {#2}="c10"}
  \IfValueTF{#3}
    {\POS{{#3}="c11"}}
    {\POS{"c01"+"c10"-"c00"="c11"}}
  \fancypullbackcore
    {\IfValueTF{#4}{#4}{0.2cm}}
    {\IfValueTF{#5}{#5}{0.6}}
  \restore}

\NewDocumentCommand\fancypullbackhops{m m g o o}{
  \IfValueTF{#3}
    {\pullbackcorewrapper{[#1]}{[#2]}{[#3]}[#4][#5]}
    {\pullbackcorewrapper{[#1]}{[#2]}[#4][#5]}}

\newdir{_|_}{^\dir{|-}}

\newdir{ >}{{}*!/-7pt/@{>}}
\newdir{m}{->}
\newcommand{\xycenter}[1]{\begin{gathered}\xymatrix{#1}\end{gathered}}

\makeatletter
\newcommand{\DeclareAbbrevation}[2]{\newcommand{#1}{\@ifnextchar{.}{#2}{#2.\@\xspace}}}
\makeatother

\DeclareAbbrevation{\ie}{i.e}
\DeclareAbbrevation{\eg}{e.g}
\DeclareAbbrevation{\cf}{cf}
\DeclareAbbrevation{\etc}{etc}
\DeclareAbbrevation{\resp}{resp}
\DeclareAbbrevation{\etal}{et al}
\DeclareAbbrevation{\ibid}{ibid}

\newcommand{\defeq}{=_{\operatorname{def}}}
\newcommand{\co}{\colon}

\newcommand{\op}{{\operatorname{op}}}

\newcommand{\cal}[1]{\mathcal{#1}}

\newcommand{\N}{\mathbb{N}}

\newcommand{\Ex}{\operatorname{Ex}}

\UseAllTwocells

\DeclarePairedDelimiter\bracks\lbrack\rbrack

\DeclarePairedDelimiter\angles\langle\rangle
\DeclarePairedDelimiter\braces\lbrace\rbrace
\DeclarePairedDelimiterX\set[2]\lbrace\rbrace{#1 \mathrel{\delimsize\vert} #2}

\newcommand{\cc}{\mathbin{\circ}}
\newcommand{\id}{\operatorname{id}}

\newcommand{\cod}{\operatorname{cod}}
\newcommand{\dom}{\operatorname{dom}}
\newcommand{\colim}{\operatorname{colim}}
\newcommand{\Id}{\operatorname{Id}}

\newcommand{\hatcirc}{\mathbin{\widehat{\circ}}}

\newcommand{\liftl}[1]{{\fourIdx{\pitchfork}{}{}{}{\smash{#1}\vphantom{I}}}}
\newcommand{\liftr}[1]{{\fourIdx{}{}{\pitchfork}{}{\smash{#1}\vphantom{I}}}}

\newsavebox{\mybox}
\newcommand{\scaledreflect}[1]{%
  \ThisStyle{\ifmmode%
    \savebox{\mybox}{$\SavedStyle#1$}%
    \reflectbox{\usebox{\mybox}}%
  \else%
    \savebox{\mybox}{#1}%
    \reflectbox{\usebox{\mybox}}%
  \fi%
}}

\newcommand{\obackslash}{\mathbin{\scaledreflect{\oslash}}}

\newcommand{\hatotimes}{\mathbin{\widehat{\otimes}}}
\newcommand{\hatobackslash}{\mathbin{\widehat{\obackslash}}}

\makeatletter
\def\ignorespacesandallpars{%
  \@ifnextchar\par
    {\expandafter\ignorespacesandallpars\@gobble}%
    {}%
}
\makeatother

\newcommand{\note}[3][]{\def\auth{#1}\textcolor{#2}{{[\ifx\auth\empty\else\auth: \fi{#3}]}}}

\newcommand{\notec}{\note{red}}

\newcommand{\triv}{\textup{triv}}

\newcommand{\Adj}{\operatorname{Adj}}

\newcommand{\fib}{\twoheadrightarrow}
\newcommand{\cof}{\rightarrowtail}
\newcommand{\trivfib}{\mathrel{\mathrlap{\hspace{0pt}\raisebox{5pt}{$\scriptscriptstyle\triv$}}\mathord{\twoheadrightarrow}}}
\newcommand{\trivcof}{\mathrel{\mathrlap{\hspace{0pt}\raisebox{5pt}{$\scriptscriptstyle\triv$}}\mathord{\rightarrowtail}}}
\newcommand{\we}{\mathrel{\mathrlap{\hspace{1pt}\raisebox{4pt}{$\scriptscriptstyle\sim$}}\mathord{\rightarrow}}}

\setlist[enumerate]{label=(\roman*)}

%% file: formatting-christian.tex
\newtheorem{theorem}{Theorem}[section]
\newtheorem*{theorem*}{Theorem}
\crefname{theorem}{Theorem}{Theorems}

\defthm{corollary}{corollaries}{Corollary}{Corollaries}
\defthm{lemma}{lemmata}{Lemma}{Lemmata}
\defthm{proposition}{propositions}{Proposition}{Propositions}
\defthm{exercise}{exercises}{Exercise}{Exercises}

\theoremstyle{definition}

\defthm{definition}{definitions}{Definition}{Definitions}

\defthm{remark}{remarks}{Remark}{Remarks}
\defthm{example}{examples}{Example}{Examples}
\defthm{question}{questions}{Question}{Questions}
\defthm{assumption}{assumptions}{Assumption}{Assumptions}

\crefname{section}{Section}{Sections}
\crefname{subsection}{Subsection}{Subsections}

\crefname{equation}{}{}
\numberwithin{equation}{section}